\newif\ifarxive

\arxivetrue

\ifarxive
	\documentclass[a4paper]{article}
	\usepackage{hyperref}
	\usepackage{mathptmx, amsmath, amssymb, amsfonts, amsthm, mathptmx, enumerate, color}
\else
	\documentclass[12pt,twoside,reqno]{amsart}
	\usepackage[colorlinks=true,citecolor=blue]{hyperref}
	\usepackage{mathptmx, amsmath, amssymb, amsfonts, amsthm, mathptmx, enumerate, color}
\fi

\usepackage{graphicx}
\usepackage{epstopdf}
\usepackage[ruled,linesnumbered]{algorithm2e}
\DeclareMathAlphabet{\mathcal}{OMS}{zplm}{m}{n}
\usepackage[noadjust]{cite}
\usepackage{tikz,tikz-3dplot}
\usetikzlibrary{patterns}

\ifarxive
	\newtheorem{theorem}{Theorem}
	\newtheorem{corollary}[theorem]{Corollary}
	\newtheorem{lemma}[theorem]{Lemma}
	\newtheorem{proposition}[theorem]{Proposition}
	\theoremstyle{definition}
	\newtheorem{definition}[theorem]{Definition}
	\newtheorem{example}[theorem]{Example}
	\newtheorem{remark}[theorem]{Remark}
\else
	\newtheorem{theorem}{Theorem}[section]
	
	\newtheorem{lemma}{Lemma}[section]
	\newtheorem{proposition}{Proposition}[section]
	\theoremstyle{definition}
	
	\newtheorem{example}{Example}[section]
	
	\numberwithin{equation}{section}
\fi

\newcommand{\R}{\mathbb{R}}

\renewcommand{\H}{\mathcal{H}}

\renewcommand{\P}{\mathcal{P}}
\newcommand{\D}{\mathcal{D}}
\newcommand{\T}{\mathcal{T}}

\newcommand{\Int}{{\rm int  \,}}
\newcommand{\conv}{{\rm conv  \,}}
\newcommand{\Cl}{{\rm cl  \,}}
\newcommand{\bd}{{\rm bd  \,}}
\newcommand{\vertex}{{\rm vert\,}}
\newcommand*{\colonequals}{\mathrel{\vcenter{\baselineskip0.5ex%
\lineskiplimit0pt\hbox{\scriptsize.}\hbox{\scriptsize.}}}=}

\newcommand{\abstracttext}{
	A polyhedral approximation of a convex body can be calculated by solving approximately an associated multiobjective convex program (MOCP).
	An MOCP can be solved approximately by Benson type algorithms, which compute outer and inner polyhedral approximations of the problem's upper image.
	Polyhedral approximations of a convex body can be obtained from polyhedral approximations of the upper image of the associated MOCP.
	We provide error bounds in terms of the Hausdorff distance for the polyhedral approximations of a convex body in dependence of the stopping criterion of the primal and dual Benson type algorithms which are applied to the associated MOCP.
}
\ifarxive
\newcommand{\keywordtext}{convex body, polyhedral approximation, multiobjective programming, projection problem, duality }
\else
\newcommand{\keywordtext}{Convex body; Polyhedral approximation; Multiobjective programming; Projection problem; Duality. }
\fi

\ifarxive
	\title{On the approximation error for approximating convex bodies using multiobjective optimization}
	\author{Andreas L\"{o}hne$^1$, Fangyuan Zhao$^1$, Lizhen Shao$^2$}
\begin{document}
	
	\footnotetext[1]{Friedrich Schiller University Jena, Faculty of Mathematics and Computer Science, Jena, Germany, [andreas.loehne$|$fangyuan.zaho]@uni-jena.de}
	\footnotetext[2]{University of Science and Technology Beijing, School of Automation and Electrical Engineering, Beijing, China, lshao@ustb.edu.cn}
	\maketitle

	\begin{abstract} \abstracttext
	
	\medskip

	\noindent
	{\bf MSC 2010 Classification:}
	\end{abstract}
	
\else

	\begin{document}
	\bibliographystyle{unsrt}
	\setcounter{page}{1}
	
	\vspace*{1.0cm}
	\title[Approximating convex bodies using multiobjective optimization]{On the approximation error for approximating convex bodies using multiobjective optimization}
	\author[A. L\"{o}hne, F. Zhao, L. Shao]{Andreas L\"{o}hne$^{1,*}$, Fangyuan Zhao$^{1}$, Lizhen Shao$^2$}
	\maketitle
	\vspace*{-0.6cm}

	\begin{center}
	{\footnotesize {\it

	$^1$Department of Mathematics, Friedrich Schiller University Jena, Jena, Germany\\
	$^2$School of Automation and Electrical Engineering, University of Science and Technology Beijing, Beijing, China%

	}}\end{center}
	
	\vskip 4mm {\small\noindent {\bf Abstract.}
	\abstracttext
	
	}

	\renewcommand{\thefootnote}{}
	\footnotetext{ $^*$Corresponding author.
	\par
	E-mail addresses: andreas.loehne@uni-jena.de (A. L\"{o}hne), fangyuan.zhao@uni-jena.de (F. Zhao), lshao@ustb.edu.cn (L. Shao).
	\par
	 }
\fi

%
%

\section{Introduction}

We aim to compute polyhedral approximations $Y_{approx}$ of a convex body $Y$. This task is referred to by the following {\em approximate convex projection problem},
\begin{subequations}
\begin{align}
     {\rm approximate} \quad Y \colonequals \{G x \mid \; g_1(x) \leq 0, \ldots,g_m(x) \leq 0\},  \tag*{\rm (CPP)}\label{CPP}
\end{align}
\end{subequations}
where $G$ is assumed to be a full row rank matrix in $\R^{q\times n}$ and the convex functions $g_i:\R^n \to \R\cup \{\infty\}$ are given. We denote by $X \colonequals\{x \in \R^n \mid g_1(x) \leq 0, \ldots,g_m(x) \leq 0 \}$ the feasible set of \ref{CPP}.
The term {\em projection problem} is used because $Y$ can be seen as a projection of the set
$$\{(x,y) \in \R^n \times \R^q \mid y = Gx,\; x \in X\}$$
into the space of $y$-components. A rigorous solution concept for the polyhedral approximation problem has been introduced in \cite{lohne2016}. Moreover, it has been shown there that polyhedral projection is equivalent to multiobjective linear programming (MOLP). Here we reduce our considerations to finding polyhedral approximations of $Y$. To this end we solve an associated multiobjective convex program (MOCP). The object of study is the approximation error of (CPP) in dependence of the approximation error in (MOCP).

The feasible set $X$ is a convex set. Moreover, $X$ is assumed to be a {\em convex body}, i.e. a convex and compact set having non-emtpy interior. Since $Y$ is the image of $X$ under the linear map $G$, it is also a convex body.

The following multiobjective convex program is associated to \ref{CPP} (see \cite{lohne2016} for a polyhedral prototype of this idea, and \cite{Shao18} for an extension to the non-polyhedral convex case):
\begin{subequations}
\begin{align}
     {\rm min} \quad \Gamma(x) = \left(
\begin{array}{c}
Gx\\
-e^TGx
\end{array}
\right) \quad w.r.t. \quad \leq_{\R^{q+1}_+} \quad s.t. \quad x \in X,  \tag*{\rm (MOCP)} \label{MOCP}
\end{align}
\end{subequations}
where $e$ denotes the vector where all entries are one and $\R^{q+1}_+ \colonequals \{x \in \R^{q+1}: x \geq 0\}$. Problem \ref{MOCP} is feasible as $X$ is non-empty. The {\em image} of the feasible set is defined by $\Gamma[X] \colonequals \{\Gamma(x) \in \R^{q+1} \mid x\in X\}$. The {\em upper image} of \ref{MOCP} is the set $\P \colonequals \Cl (\Gamma[X] + \R^{q+1}_+)$. $\Gamma[X]$ is compact because it is an image of the compact set $X$ under a linear transformation. Consequently, the upper image $\P$ can be written as $\P = \Gamma[X] + \R^{q+1}_+$.

In this article we provide tight bounds for the approximation error of the polyhedral projection problem \ref{CPP} in dependence of the prescribed tolerance $\varepsilon > 0$ for the associated multiobjective convex program \ref{MOCP}. This tolerance is defined by a certain stopping criterion in the primal and dual Benson type algorithms, respectively. Solving an arbitrary multiobjective convex program with the primal and dual Benson type algorithms, we obtain, respectively, an outer and inner approximation $\P_{approx}$ of the upper image $\P$ with
$$ d_H(\P_{approx},\P) \leq \varepsilon \sqrt{q+1}.$$
Here $q$ is the dimension of $Y$ (i.e. $\P$ is of dimension $q+1$) and $d_H$ denotes the Hausdorff distance. In the framework of arbitrary multiobjective convex programs, this bound is tight for both the primal and dual algorithms.

\ref{CPP} is solved by applying a primal or dual Benson type algorithm to the associated \ref{MOCP} with tolerance $\varepsilon>0$.
For the convex body $Y$ we obtain, respectively, an outer and inner approximation $Y_{approx}$ with
$$ d_H(Y_{approx},Y) \leq \varepsilon \sqrt{q^2+q-1}.$$
This bound is tight in the sense that both the primal and dual Benson type algorithms can induce the worst case error in the projection problem \ref{CPP}.

We see that the bound for the approximation error of \ref{CPP} differs from the bound for the approximation error for the associated \ref{MOCP} by the factor $$\gamma(q) \colonequals \sqrt{q-\frac{1}{q+1}} < \sqrt{q}.$$
 For $q \geq 2$ the bound for \ref{CPP} is bigger than the one for the associated \ref{MOCP}. Related investigations were made independently in \cite{kovavcova2021convex}.

If the functions $g_i$, $i=1,\ldots,m$ are affine, then problem \ref{CPP} reduces to a polyhedral projection problem. The equivalence between polyhedral projection and multiobjective linear programming is used in {\em bensolve tools} (a software package for polyhedral calculus and related problems, see \cite{bt}), where the MOLP solver {\em bensolve} \cite{bensolve} is utilized to solve polyhedral projection problems. The options of {\em bensolve tools} allow to control the tolerance $\varepsilon > 0$ of the associated MOLP. However, the resulting approximation error for the polyhedral approximation problems was not specified so far. This means our result is new even for the polyhedral case. 

The results in this article are based on so-called Benson type algorithms, insofar we relate our error bounds to the typical stopping criteria of these algorithms. These methods are named after Harold P. Benson, who came up with an outer approximation algorithm for MOLP problems in 1998 \cite{benson1998}. Based on geometric duality \cite{Heyde08}, a dual variant of Benson's algorithm for MOLP was proposed in 2012 \cite{ehrgott2012dual}.
However, the basic idea of these algorithms can be found earlier in the literature. For instance, in 1982, Mukhamediev \cite{Muk82} used ideas similar to the dual Benson type algorithm to solve certain global optimization problems. In 1992, Kamenev \cite{kamenev1992class} presented two adaptive methods to approximate convex bodies by polyhedra, which are similar to the primal and dual Benson type algorithms. The dual Benson type algorithm is also similar to the convex hull method by Lassez \& Lassez from 1992 \cite{lassez1992quantifier}. Benson's approach to solve MOLPs was extended to MOCP, for instance, in \cite{chan2014, dorfler2020benson, Ehrgott11, Lohne14, rennen2011}.

This article is organized as follows. Section \ref{sec:notation} is devoted to notation and preliminary results. In Section \ref{sec:primal}, the error bound for \ref{CPP} based on the primal Benson type algorithm is shown. The error bound for \ref{CPP} based on the dual Benson type algorithm is provided in Section \ref{sec:dual}. In Section \ref{sec:conclusion}, we draw some conclusions.

\section{Preliminaries}\label{sec:notation}

The following notation will be used throughout this paper. For a set $A \subseteq \R^q$, we denote the boundary, interior, closure, convex hull of $A$ by $\bd A$, $\Int A$, $\Cl A$, $\conv A$. Let $a \colonequals(a_1,\ldots,a_{q+1})^T\in \R^{q+1}$, $A \subseteq \R^{q+1}$. The vector, which omits the last component of $a$, is denoted by $\pi(a)$, i.e. for $a=(a_1,\ldots,a_{q+1})^T$ we have $\pi(a)=(a_1,\ldots,a_q)^T$. For $A \subseteq \R^{q+1}$ we set $\pi[A] \colonequals \{\pi(a): a \in A\} \subseteq \R^q$. $I$ denotes the identity matrix and $e^i$ is the $i$-th unit vector. The hyperplane
$$H \colonequals \{y\in \R^{q+1}: e^T y = 0\}$$
is frequently used.

The image of the feasible set $\Gamma[X]$, the upper image $\P$ in problem \ref{MOCP} and the convex set $Y$ in \ref{CPP} are related in the following sense.

\begin{proposition}\label{prop:MOCP}
For problem {\rm\ref{MOCP}}, the following statements hold.
\begin{enumerate}[{\rm(i)}]
  \item $\Gamma[X] = \P \bigcap H$,
  \item $Y=\pi[\Gamma[X]]$.
\end{enumerate}
\end{proposition}

\begin{proof}
(i) For all $y \in \Gamma[X]$, $e^T y = 0$ holds. This implies $\Gamma[X] \subseteq H$. We have $\Gamma[X] \subseteq \Gamma[X]+\R^{q+1}_+ = \P$. Thus  $\Gamma[X] \subseteq \P \bigcap H$. Now let $s\in \P \bigcap H$. Then there exist $x \in X$ and $t \in \R^{q+1}_+$ such that $s = \Gamma(x) + t$ and $e^T s= e^T \Gamma(x) + e^T t = e^T t = 0$. This yields $t=0$ and thus $s \in \Gamma[X]$.

(ii) Follows from the definitions of $\pi(\cdot)$, $\Gamma[X]$ and $Y$.
\end{proof}

A Benson type algorithm computes an \textit{outer approximation} $\P_{outer}$ (resp. \textit{inner approximation} $\P_{inner}$). This means $\P_{outer}$ (resp. $\P_{inner}$) is a superset (resp. subset) of $\P$ and close  to $\P$ in some sense. We next show that a set $\P_{outer} \supseteq \P$ yields a set $Y_{outer} \supseteq Y$ and a set $\P_{inner} \subseteq \P$ yields a set $Y_{inner} \subseteq Y$. Later we show that $Y_{outer}$ and $Y_{inner}$ are close to $Y$ in some sense.

\begin{proposition}\label{prop:appY}
	For $\P_{outer} \supseteq \P$ and $\P_{inner} \subseteq \P$ we have
\begin{enumerate}[{\rm(i)}]
  \item $Y_{outer} \colonequals \pi[\P_{outer} \bigcap H] \supseteq Y$,
  \item $Y_{inner} \colonequals \pi[\P_{inner} \bigcap H] \subseteq Y$.
\end{enumerate}
\end{proposition}

\begin{proof}
(i) From $\P_{outer} \supseteq \P$ and Prop. \ref{prop:MOCP} (i), we get $\P_{outer} \bigcap H \supseteq \P \bigcap H = \Gamma[X]$. Thus $\pi[\P_{outer} \bigcap H] \supseteq \pi[\Gamma[X]]$. Prop. \ref{prop:MOCP} (ii) yields $Y_{outer} \supseteq  Y$.

(ii) From $\P_{inner} \subseteq \P$ and Proposition \ref{prop:MOCP} (i), we get $\P_{inner} \bigcap H \subseteq \P \bigcap H = \Gamma[X]$. Thus $\pi[\P_{inner} \bigcap H] \supseteq \pi[\Gamma[X]]$. Prop. \ref{prop:MOCP} (ii) yields $Y_{inner} \subseteq  Y$.
\end{proof}

We next recall the {\em geometric dual problem} of \ref{MOCP} (cf. \cite{Lohne14}). For $t \in \R^{q+1}$, let
$$w(t) \colonequals \left(t_1,\ldots, t_{q}, 1-\sum_{i=1}^{q}t_{i}\right)^T.$$
The geometric dual of \ref{MOCP} is given by
\begin{subequations}
\begin{align}
     {\rm max} \quad D^*(t) \quad w.r.t. \;\leq_K \quad s.t. \quad \: w(t) \geq 0,  \tag*{(\rm MOCP*)} \label{dual_MOCP}
\end{align}
\end{subequations}
where
 $$D^*(t) \colonequals \left(t_1, \dots, t_{q}, \inf \limits_{x \in X}\left[w(t)^T\Gamma(x)\right]\right)^T,$$
  $K\colonequals \R_+(0, 0, \ldots, 0, 1)^T$ and $\leq_K$ is the ordering induced by the cone $K$, that is, $y^* \leq_K v^*$ iff $v^*-y^*\in K$.
The {\em lower image} of \ref{dual_MOCP} is defined as $\D \colonequals D^*[\T] - K$, where $\T \colonequals \{t \in \R^{q+1} \mid w(t) \geq 0\}$ is the feasible region of \ref{dual_MOCP}.
For $y,y^*\in \R^{q+1}$, we consider the two following hyperplane-valued maps:
\begin{eqnarray*}
{\H}:  \R^{q+1} \rightrightarrows \R^{q+1}, & & {\H}(y^*)  \colonequals \{y \in \R^{q+1} : \varphi(y,y^*)=0\},\\
{\H}^*:\R^{q+1} \rightrightarrows \R^{q+1}, & & {\H}^*(y)\colonequals \{y^* \in \R^{q+1} : \varphi(y,y^*)=0\},
\end{eqnarray*}
where $\varphi$ is a bi-linear {\em coupling function} defined as
\begin{equation*}
\varphi(y,y^*)= \sum_{i=1}^{q}y_i y^*_i+y_{q+1}\left(1-\sum_{i=1}^{q} y^*_i\right) - y^*_{q+1}.
\end{equation*}
For the {\em duality mapping}
$$\Psi :2^{\R^{q+1}} \rightarrow 2^{\R^{q+1}},\quad \Psi(F^*)\colonequals\bigcap\limits_{y^*\in F^*} \H(y^*) \cap \P,$$ the following {\em geometric duality} relation between $\P$ and $\D$ holds:

\begin{theorem}[\cite{heyde2013}]\label{thm:duality}
$\Psi$ is an inclusion reversing one-to-one mapping between the set of all $K-$maximal exposed faces of $\D$ and the set of all weakly minimal exposed faces of $\P$. The inverse map is given by
$$\Psi^{-1}(F)=\bigcap \limits_{y\in F} \H^*(y) \cap \D.$$
\end{theorem}

Finally, we recall two kinds of scalarizations for \ref{MOCP}. For $w \in \R^{q+1}$, $e^T w = 1$, the {\em weighted sum scalarization} is the convex program
\begin{subequations}
\begin{align}
     \min_x w^T \Gamma(x) \quad \text{s.t.} \quad g(x) \leq 0  \tag*{$($P$_1(w))$}. \label{P1}
\end{align}
\end{subequations}
The Lagrangian dual program of \ref{P1} is
\begin{subequations}
\begin{align}
     \max_{u \in \R^m} \inf_{x\in\R^n} \left[w^T\Gamma(x)+u^Tg(x)\right] \quad \text{s.t.} \quad u \geq 0  \tag*{$($D$_1(w))$}, \label{D1}
\end{align}
\end{subequations}
where $g=(g_1,\ldots,g_m)^T$ and $0 \cdot \infty \colonequals 0$. An optimal solution of \ref{P1} is a weak minimizer of \ref{MOCP} (see e.g. \cite{jahn11}).

The other kind of scalarization for \ref{MOCP} is known under different names in the literature, among them: {\em translative scalarization}, {\em Tammer-Weidner scalarization} or {\em Pascoletti-Serafini scalarization}. It depends on a parameter vector $v \in \R^{q+1}$, which typically does not belong to $\Int \P$. It is given by the convex program
\begin{subequations}
\begin{align}
     \min_{x\in \R^n,\, z\in \R} z \quad \text{s.t.}\quad  g(x)\leq 0,\; \Gamma(x) -z e -v \leq 0  \tag*{$($P$_2(v))$}. \label{P2}
\end{align}
\end{subequations}
Its Lagrangian dual problem is
\begin{subequations}
\begin{align}
     \max_{u\in \R^m \atop w \in \R^{q+1}} \inf_{x\in \R^n} \left[w^T\Gamma(x)+u^Tg(x)\right]-w^Tv  \quad \text{s.t.}\quad  u \geq 0,\; w \geq 0,\; w^Te = 1  \tag*{$($D$_2(v))$}. \label{D2}
\end{align}
\end{subequations}

In this paper, we will use the Hausdorff distance to measure the approximation error. For nonempty sets $M, N \subseteq \R^q$ with $M \subseteq N$ it can be expressed as
\begin{equation*}\label{eq:twosideH}
d_H(M,N)\colonequals \sup_{b \in N} \inf_{a \in M} \|a-b\|_2.
\end{equation*}

\begin{lemma}\label{lemma:dis}
Consider nonempty sets $M, N \subseteq \R^q$ and a vector $d\in \R^q$ such that $N + \{d\} \subseteq M \subseteq N$. Then,
$$d_H(M,N) \leq \|d\|_2.$$
\end{lemma}
\begin{proof}
For arbitrary $u\in N$, we get $u + d \in M$, i.e., there exists $v \in M$ such that $u + d = v$. Since $M \subseteq N$, we get $d_H(M,N) =\sup\limits_{u \in N} \inf\limits_{a \in M} \|a-u\|_2 \leq \|d\|_2$.
\end{proof}

\section{Error bounds for the primal algorithm}\label{sec:primal}

The primal Benson type algorithm for a multiobjective convex program (with $q+1$ objectives) computes a shrinking sequence of outer polyhedral approximations $\P_k$ for $\P$, where $k$ is the iteration index. The optimal values $y_i$ of \ref{P1} for $w$ being the unit vectors $e^i$, for $i=1,\ldots,q+1$, are computed and the initial outer polyhedral approximation of $\P$ is set to $\P_0 = \{y\} + \R^{q+1}_+$. In each iteration, an arbitrary vertex $v$ of current outer approximation $\P_k$ is chosen and a point $o$ of $\Gamma[X]$ is obtained by solving \ref{P2}. If $o - v \leq \epsilon e$, for a prescribed tolerance $\epsilon > 0$, the algorithm continues to check the other vertices of $\P_k$. If $o - v \leq \epsilon e$ holds for all vertices of $\P_k$, the algorithm terminates. Otherwise, for some $v$ and corresponding $o$ with $o - v \not\leq \epsilon e$, a supporting half-space $H_+ \supseteq \P$ of $\P$ at the point $o$ is calculated by solving \ref{D2}. The new outer approximation is updated to $\P_{outer} \leftarrow \P_{outer} \bigcap H_+$. Algorithm \ref{alg:primal} provides a pseudo code of a simplified version of the primal Benson type algorithm.

\begin{algorithm}[ht]\label{alg:primal}
\caption{Simplified primal Benson type algorithm for MOCP (\cite{Lohne14})}
\LinesNumbered
\KwIn{$\epsilon$, $\P_0$}
\KwOut{An outer approximation $\P_{outer}$ of $\P$}
$V \leftarrow \emptyset$; $\P_{outer}\leftarrow \P_0$;\\
compute the vertex set $\bar V$ of $\P_{outer}$;\\
\While{$\bar V \backslash V \neq \emptyset$}
{
choose $v \in \bar V \backslash V$;\\
$(x,z,u,w) \leftarrow$ \textrm{solve}(\ref{P2},\ref{D2});\\
\uIf {$z > \epsilon$}
{
$\P_{outer} \leftarrow \P_{outer}  \bigcap \{y\in \R^{q+1}:\varphi(y,D^*(w)) \geq 0\}$;\\
update the vertex set $\bar V$ of $\P_{outer}$;\\
}
\Else
{
 $V \leftarrow V \bigcup \{v\}$;\\
}

}
\textbf{return} $\P_{outer}$\\
\end{algorithm}

Algorithm \ref{alg:primal} applied to the \ref{MOCP} associated to \ref{CPP} yields an outer approximation $\P_{outer}$ for the upper image $\P$ of \ref{MOCP} and, by Proposition \ref{prop:appY}, an outer approximation $Y_{outer}$ of $Y$. The result of the following proposition can be also found in \cite{Lohne14}, formulated in terms of {\em finite $\varepsilon$-infimizers}. For the convenience of the reader we present here a short and direct proof.

\begin{proposition}\label{remark_primal}
	For an arbitrary instance of problem \ref{MOCP} and some given $\varepsilon > 0$, let $\P_{outer}$ be the result of Algorithm \ref{alg:primal}. Then,
	$$ \P_{outer} + \varepsilon \{e\}\subseteq \P.$$
\end{proposition}

\begin{proof}
Let $v \in \P_{outer}$. Then $v$ can be expressed by the vertices $v^1,\ldots,v^k$ of $\P_{outer}$ and a point $c \in \R^{q+1}_+$ as
$$ v = \sum_{i=1}^{k} \lambda_i v^i + c,\quad \lambda_1,\ldots,\lambda_k \geq0,\quad \sum_{i=1}^{k} \lambda_i=1.$$
By the stopping criterion of Algorithm \ref{alg:primal}, for each $v^i$, $i=1,\ldots,k$, there exists $o^i \in \Gamma[X]$ such that $o^i - v^i \leq \epsilon e$. For $o \colonequals \sum\limits_{i=1}^{k}\lambda_i o^i \in \Gamma[X]$ we have
\begin{equation*}
	o-v =  \sum\limits_{i=1}^{k}\lambda_i (o^i - v^i) -c \leq \varepsilon e - c \leq \varepsilon e.
\end{equation*}
Thus $v +\varepsilon e$ belongs to $\Gamma[X]+\R^{q+1} =\P$, which proves the claim.
\end{proof}

We start with a bound for \ref{MOCP}, see e.g. Remark 3.4 in \cite{dorfler2020benson}.

\begin{proposition}\label{prop:pouter}
For an arbitrary instance of problem \ref{MOCP} with $q+1$ objectives, let $\P$ be the upper image and $\P_{outer}$ be the result of Algorithm \ref{alg:primal} for some given $\varepsilon > 0$. Then
$$d_H(\P_{outer},\P) \leq \epsilon\sqrt{q+1}.$$
\end{proposition}
\begin{proof}
By Proposition \ref{remark_primal} and Lemma \ref{lemma:dis}, we get $d_H(\P_{outer},\P) \leq \|\epsilon e\|_2 = \epsilon \sqrt{q+1}$.
\end{proof}

This bound is tight as shown by the following example.

\begin{example}\label{ex_p1} Let $\varepsilon = \frac{1}{q+1}$, $\P = \{x \in \R^{q+1} \mid x + \frac{1}{q+1} e \geq 0,\; e^T (x+ \frac{1}{q+1} e) \geq 1\}$. Then Algorithm \ref{alg:primal} terminates with the initial outer approximation $\P_{outer}=\{- \frac{1}{q+1} e\}+ \R^{q+1}_+$ and we have $d_H(\P_{outer},\P) = \epsilon\sqrt{q+1}$.
\end{example}

Now we prove the error bound for \ref{CPP} solved with Algorithm \ref{alg:primal}.

\begin{theorem}\label{prop:primaldis}
	Let an arbitrary instance of \ref{CPP} be given. Let $\P_{outer}$ be the result of Algorithm \ref{alg:primal} applied to the associated \ref{MOCP} for some given $\varepsilon > 0$ and let $Y_{outer} \colonequals \pi[\P_{outer} \bigcap H]$. Then $Y_{outer} \supseteq Y$ and
$$d_H(Y_{outer},Y) \leq \epsilon\sqrt{q^2+q-1}.$$	
\end{theorem}
\begin{proof}	
Let $\bar v$ be an arbitrary point in $Y_{outer}$.
Then $v \colonequals(\bar v^T,-e^T\bar v)^T \in \P_{outer} \bigcap H$. By Proposition \ref{remark_primal}, we have $v+\varepsilon e \in \P$. Thus there exists $o \in \Gamma[X]$ with $o \leq v+ \epsilon e$. Since $\Gamma[X]\subseteq H$, we have $e^T o = 0$. By Proposition \ref{prop:MOCP} (ii), we have $\bar o \colonequals \pi(o) \in Y$.
We have $\bar o_i \leq \epsilon + \bar v_i$, $i=1,\ldots,q$ and $o_{q+1}=-e^T \bar o \leq \epsilon + v_{q+1} = \epsilon -e^T\bar v$.
Thus the point $\bar o - \bar v$ belongs to the polytope $T \colonequals \{t\in\R^{q} \mid -e^T t \leq \epsilon,\; t \leq \epsilon e\}$. The vertex set of $T$ is
$$\vertex T = \{(\epsilon,\ldots, \epsilon)^T,(-\epsilon q, \epsilon , \ldots, \epsilon)^T,(\epsilon,-\epsilon q, \epsilon, \ldots, \epsilon)^T,\ldots,(\epsilon, \ldots, \epsilon ,-\epsilon q)^T\}.$$
 Thus $\|\bar u-\bar v\|_2 \leq \max\{\|t\|_2:t\in \vertex T\}= \epsilon \sqrt{q^2+q-1}$.
Since $Y_{outer} \supseteq Y$ and for arbitrary $\bar v \in Y_{outer}$ there exists $\bar u \in Y$ with $\|\bar u-\bar v\|_2 \leq \epsilon \sqrt{q^2+q-1}$, we conclude that $d_H(Y_{outer},Y) \leq \epsilon\sqrt{q^2+q-1}$.
\end{proof}

The following example shows that the upper bound of Theorem \ref{prop:primaldis} is tight for all $q \geq 2$, see also Figure \ref{fig:p}.

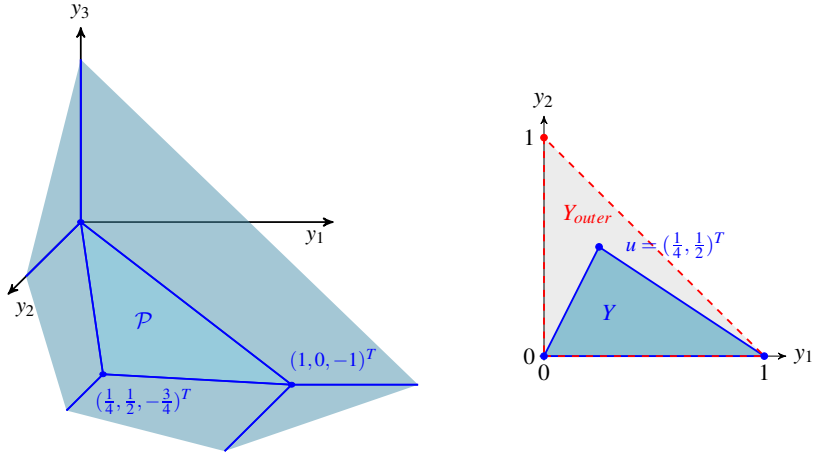
\begin{figure}[hbt]
\resizebox{0.49\textwidth}{!}{%
\centering
\begin{tikzpicture}[x=45,z=-13,y=35,scale=2.1]
\coordinate (NE) at (1.5,1.3);
\coordinate (SW) at (-.5,-1.5);
\draw[white,fill=white] (NE) circle (0.01);
\draw[white,fill=white] (SW) circle (0.01);
\coordinate (NULL) at (0,0,0);
\coordinate (X) at (1.2,0,0);
\coordinate (Y) at (0,0,1.2);
\coordinate (Z) at (0,1.2,0);
\coordinate (P1) at (0,0,0);
\coordinate (P2) at (1,-1,0);
\coordinate (P3) at (0.25,-0.75,0.5);

\coordinate (Q1) at (0,0,0);
\coordinate (Q2) at (1,-1,0);
\coordinate (Q3) at (0,-1,1);

\coordinate (B1) at (0,1,0);
\coordinate (B2) at (0,0,0.9);
\coordinate (B3) at (0.25,-0.75,1.1);
\coordinate (B4) at (1,-1,1.1);
\coordinate (B5) at (1.6,-1,0);
\draw[thick,->,>=stealth'] (NULL) -- (X) node[anchor=north east]{$y_1$};
\draw[thick,->,>=stealth'] (NULL) -- (Y) node[anchor=north west]{$y_2$};
\draw[thick,->,>=stealth'] (NULL) -- (Z) node[anchor=south]{$y_3$};

\draw[blue,thick,fill opacity=0.5,fill=darkgray!40!cyan] (P1) -- (P2) -- (P3) -- cycle;
\draw[blue,thick,fill opacity=0.5,fill=darkgray!50!cyan] (B1) -- (P1) -- (P2) -- (B5);
\draw[blue,thick,fill opacity=0.5,fill=darkgray!50!cyan] (B4) -- (P2) -- (B5);
\draw[blue,thick,fill opacity=0.5,fill=darkgray!50!cyan] (B3) -- (P3) -- (P2) -- (B4);
\draw[blue,thick,fill opacity=0.5,fill=darkgray!50!cyan] (B2) -- (P1) -- (P3) -- (B3);
\draw[blue,thick,fill opacity=0.5,fill=darkgray!50!cyan] (B2) -- (P1) -- (B1);

\draw[blue, fill=blue] (P1) circle (0.015);
\draw[blue, fill=blue] (P2) circle (0.015);
\draw[blue, fill=blue] (P3) circle (0.015);
\draw[blue] (.3,-.6) node {$\P$};
\draw[blue] (.3,-1.1) node {\small $(\frac{1}{4},\frac{1}{2},-\frac{3}{4})^T$};
\draw[blue] (1.2,-.85) node {\small $(1,0,-1)^T$};
\end{tikzpicture}
}
\resizebox{0.49\textwidth}{!}{%
\begin{tikzpicture}[x=45,y=45, scale=2.1]
\coordinate (NE) at (1.5,1.3);
\coordinate (SW) at (-.5,-.5);
\draw[white,fill=white] (NE) circle (0.01);
\draw[white,fill=white] (SW) circle (0.01);

\draw [opacity=0,fill opacity=.5, fill=gray!30!white] (0,0) -- (1,0) -- (0,1) -- cycle;
\draw[->, >=stealth'] (-0.02,0) -- (1.1,0) node[right] {$y_1$};
\draw[->, >=stealth'] (0,-0.02) -- (0,1.1) node[above] {$y_2$};
\foreach \x/\xtext in {0,1}
   \draw (\x,0.1pt) -- (\x,-0.1pt) node[anchor=north] {$\xtext$};
\foreach \x/\xtext in {0,1}
   \draw (0.1pt,\x) -- (-0.1pt,\x) node[anchor=east] {$\xtext$};

\draw [blue,thick,fill opacity=0.5,fill=darkgray!40!cyan] (0,0) -- (1,0) -- (0.25,0.5) -- cycle;
\draw [red,thick,dashed] (0,0) -- (1,0) -- (0,1) -- cycle;
\draw[blue] (.3,.2) node {$Y$};
\draw[blue] (.6,.51) node {\small $u=(\frac{1}{4},\frac{1}{2})^T$};
\draw[red] (.2,.65) node {$Y_{outer}$};
\draw[blue, fill=blue] (0,0) circle (0.015);
\draw[blue, fill=blue] (1,0) circle (0.015);
\draw[red, fill=red] (0,1) circle (0.015);
\draw[blue, fill=blue] (.25,.5) circle (0.015);
\end{tikzpicture}
}
\caption{Illustration of Example \ref{ex:p} for $q=2$.}
\label{fig:p}
\end{figure}

\begin{example}\label{ex:p}
	For $q\geq 2$ let $Y$ be a $q$-dimensional simplex given by
	 \begin{equation*}
	     Y = \left\{x \in \R^q \,\bigg\vert\;\; x_q\geq 0,\; x_q \leq 2 x_i, \; i \in \{1,\ldots,q-1\},\; \sum_{i=1}^{q-1} x_i + \frac{3}{2} x_q \leq 1\right\}
	 \end{equation*}
	 and let $\varepsilon \colonequals \frac{1}{q+2}$.
	 The $q+1$ vertices of $Y$ are $0$, the unit vectors $e^i$ for $i \in \{ 1, \ldots, q-1\}$ and the point
	 $$ u=\left( \frac{1}{q+2}, \ldots ,\frac{1}{q+2},\frac{2}{q+2} \right) ^T.$$
	The upper image $\P$ has a V-represenatiaion with extremal directions $d^i=e^i$ for $i\in \{1,\ldots,q+1\}$ and vertices	
	$$ \begin{pmatrix} 0 \\ 0 \end{pmatrix},\;
	   \begin{pmatrix} \phantom{-}e^1 \\ -1 \end{pmatrix},\;\ldots,\;
	   \begin{pmatrix} \phantom{-}e^{q-1} \\ -1 \end{pmatrix},\;
	   \begin{pmatrix} \phantom{-}u  \\ -e^T u \end{pmatrix}.
	   $$
	 Since $-1 < -e^T u$, the initial outer approximation computed by the primal Benson type algorithm (for $q \geq 2$) is
	 $$ \P_0 = \{-e^{q+1}\}+ \R^{q+1}.$$
	 The vertex $v=-e^{q+1}$ of $\P_0$ is chosen by the algorithm and we obtain $z=\frac{1}{q+1} > \varepsilon$. Since $v + z e$ is a convex combination of the vertices of $\P$ with positive coefficients
	 $$\lambda_1=\frac{1}{2(q+1)},\; \lambda_2=\frac{1}{2(q+1)},\; \ldots,\;\lambda_q=\frac{1}{2(q+1)},\;\lambda_{q+1}= \frac{q+2}{2(q+1)},$$
	 a cut with the hyperplane $H$ is made by the algorithm. The outer approximation $\P_1$ has the vertices
   $$\begin{pmatrix} 0 \\ 0 \end{pmatrix},\;
   \begin{pmatrix} \phantom{-}e^1 \\ -1 \end{pmatrix},\;\ldots,\;
   \begin{pmatrix} \phantom{-}e^{q-1} \\ -1 \end{pmatrix},\;
   \begin{pmatrix} \phantom{-}e^{q} \\ -1 \end{pmatrix}.
   $$
   The only vertex of $\P_1$ which is not in $\P$ is the last one in the list. We have
   $$ \begin{pmatrix} \phantom{-}u  \\ -e^T u \end{pmatrix} \leq \begin{pmatrix} \phantom{-}e^{q} \\ -1 \end{pmatrix} +  \varepsilon e.$$
   Thus, the algorithm terminates. The resulting outer approximation $Y_{outer}$ of $Y$ is the convex hull of
   $0$ and $e^i$ for $i\in \{1,\ldots,q\}$. For the approximation error we get
   $$ d_H(Y_{outer},Y) = \| e^q - u \| = \varepsilon \sqrt{q^2+q-1}.$$
\end{example}

\section{Error Bounds for the Dual Benson type Algorithm}\label{sec:dual}

The dual Benson type algorithm begins with an initial outer polyhedral approximation $\D_0$ of $\D$. Let $\bar x$ be an optimal solution to $($P$_1(\bar w))$ for some $\bar w > 0$ with $e^T \bar w = 1$.
The initial outer approximation can be defined by the initial inner approximation
$$\P_0 \colonequals \{\Gamma(\bar x)\} + \R^{q+1}_+$$
of $\P$ by the geometric duality relation as
$$\D_0 \colonequals \{y^*\in \R^{q+1} \mid \forall y \in \P_0:\; \varphi(y,y^*) \geq 0\}.$$
The dual Benson type algorithm constructs a sequence of shrinking polyhedral outer approximations $\D_k$ of the lower image $\D$ of the dual problem \ref{dual_MOCP}, where $k$ is the iteration index.  In each iteration, an arbitrary vertex $t$ of the current outer approximation $\D_k$ is chosen. Then a corresponding boundary point $s \colonequals D^*(t)$ of $\D$ is computed by solving \ref{P1} for $w=w(t)$. If $t_{q+1} - s_{q+1} > \varepsilon$, for a prescribed tolerance $\epsilon>0$, a portion of the current outer approximation $\D_k$ is cut off by a hyperplane supporting $\D$ at the point $s \in \D$. Otherwise, the algorithm continues to check other vertices of $\D_k$. If $t_{q+1} - s_{q+1} \leq \varepsilon$ for all vertices of $\D_k$, the algorithm terminates. By geometric duality, we obtain an inner approximation of $\P$ as
$$ \P_{inner} = \{y\in \R^{q+1} \mid \forall y^* \in \D_{outer}:\; \varphi(y,y^*) \geq 0\}.$$
In Algorithm \ref{alg:dual} we provide a pseudo code of this dual method. Note that for $x$ being an optimal solution of \ref{P1} for $w=w(t)$, we have
$$\varphi(\Gamma(x),t) =  w(t)^T \Gamma(x) - t_{q+1} = s_{q+1} - t_{q+1}.$$
Thus the stopping criterion can be expressed by the coupling function $\varphi$.

Note that we describe here a simplified  version of the algorithm only, because many details are not relevant for investigations on error bounds. For a more detailed version as well as implementation issues the reader is referred to \cite{Lohne14}.

\begin{algorithm}[ht]\label{alg:dual}
\caption{Simplified dual Benson type algorithm for MOCP (alternative Algorithm 2 in \cite{Lohne14})}
\LinesNumbered
\KwIn{$\epsilon$, $\D_0$}
\KwOut{An inner approximation $\P_{inner}$ of $\P$}
$T \leftarrow \emptyset$;
$\D_{outer} \leftarrow \D_0$;\\
compute the vertex set $\bar T$ of $\D_{outer}$;\\
\While {$\bar T \backslash T \neq \emptyset$}
{
   choose $t \in \bar T \backslash T$\\
   $x \leftarrow$ \textit{solve} $({\rm P}_1(w(t)))$;\\
   \uIf {$\varphi(\Gamma(x),t)< -\epsilon$}
    {
        $\D_{outer} \leftarrow \D_{outer} \bigcap \{y^*\in \R^{q+1}:\varphi(\Gamma(x),y^*) \geq 0\}$;\\
        update the vertex set $\bar T$ of $\D_{outer}$;\\
    }
    \Else
    {
        $T \leftarrow T \bigcup \{t\}$;\\
    }
}
$\P_{inner} = \{y \in \R^{q+1} \mid \forall y^* \in \D_{outer}:\;\varphi(y,y^*)\geq 0\}$;\\
\textbf{return} $\P_{inner}$
\end{algorithm}

The result of the following proposition can be also found in \cite{Lohne14}, formulated in terms of {\em finite $\varepsilon$-infimizers}. To make this exposition self-contained, we give a direct proof here.

\begin{proposition}\label{prop:dual}
	Let $\P_{inner}$ be the result of Algorithm \ref{alg:dual} for some given $\varepsilon > 0$. Then,
	$$ \P + \varepsilon \{e\} \subseteq \P_{inner}.$$
\end{proposition}
\begin{proof}
	Proceeding analogously to the proof of Proposition \ref{remark_primal}, we see that the stopping criterion in Algorithm \ref{alg:dual} yields
	\begin{equation}\label{eq_3}
		\D_{outer} - \varepsilon \{e^{q+1}\} \subseteq \D.
	\end{equation}
	Let $y \in \P$. By weak duality, $\varphi(y,y^*)\geq 0$ holds for all $y^* \in \D$. From \eqref{eq_3} we get
	$$ \forall y^* \in \D_{outer}:\quad \varphi(y,y^*-\varepsilon e^{q+1})\geq 0.$$
	Since $\varphi(y,y^*-\varepsilon e^{q+1})=\varphi(y+\varepsilon e,y^*)$ this implies $y+\varepsilon e \in \P_{inner}$.	
\end{proof}

As in the previous section we again start with a bound for \ref{MOCP}.

\begin{proposition}\label{prop:douter}
For an arbitrary instance of problem \ref{MOCP} with $q+1$ objectives, let $\P$ be the upper image and $\P_{inner}$ be the result of Algorithm \ref{alg:dual} for some given $\varepsilon > 0$. Then
$$d_H(\P_{inner},\P) \leq \epsilon\sqrt{q+1}.$$
\end{proposition}
\begin{proof}
This follows from Proposition \ref{prop:dual} and Lemma \ref{lemma:dis}.
\end{proof}

In the framework of arbitrary MOCP this bound is tight.

\begin{example}
	Let $S \colonequals \conv \{x^0,\,x^1,\,\ldots,\,x^{q+1}\}$ with
	$$ x^i \colonequals \frac{1}{i} \sum_{j=1}^i e^j, \qquad i\in \{1,\ldots,q+1\},$$
	$$ x^0 \colonequals \frac{1}{q+1}\sum_{i=1}^{q+1} x^i - \varepsilon e.$$
	Let $\bar \P \colonequals S+ \R^{q+1}_+$.
	We have $e^T x^i = 1$ for all $i\in \{1,\ldots,q+1\}$.
	For $\varepsilon = 0$, $x^0$ is a relative interior point of $\conv\{x^1,\ldots,x^{q+1}\}$. For sufficiently small $\varepsilon > 0$,
	the points $x^0,\ldots,x^{q+1}$ are the vertices of $\bar \P$.
		
	Consider a sufficiently small real number $\varepsilon > 0$. Choose $\bar w$ sufficiently close to $e^1$ such that $e^T \bar w = 1$ and $\bar w > 0$. Then $x^{q+1}$ is an optimal solution  to $($P$_1(\bar w))$.
	We obtain $\bar\P_0 = \{x^{q+1}\}+\R^{q+1}$. The vertices of $\bar \D_0$ are
    $$ t^1 \colonequals \begin{pmatrix} e^1 \\ \frac{1}{q+1} \end{pmatrix},\;\ldots,\;
    t^q \colonequals\begin{pmatrix} e^{q} \\ \frac{1}{q+1} \end{pmatrix},\;
    t^{q+1} \colonequals\begin{pmatrix} 0 \\ \frac{1}{q+1} \end{pmatrix}.
    $$
	We choose $t^{q+1}$. For $w=w(t^{q+1})=e^{q+1}$, $x^q$ is an optimal solution of \ref{P1} with optimal value $0$.
	For sufficiently small $\varepsilon >0$, the algorithm computes a new polyhedron
	$$ \bar\D_1 \colonequals \bar\D_0 \cap \left\{y^* \in \R^{q+1} \bigg\vert\; y^*_{q+1} \leq \frac{1}{q} \sum_{i=1}^q y^*_i \right\}.$$
	Since the vertex $t^q$ of $\bar\D_0$ also belongs to $\bar\D_1$, it is a vertex of $\bar\D_1$ and can be chosen next. We obtain $w=w(t^{q})=e^{q}$ and $x^{q-1}$ is an optimal solution of \ref{P1} with optimal value $0$. For sufficiently small $\varepsilon >0$, the algorithm computes a new polyhedron
	$$ \bar\D_2 \colonequals \bar\D_1 \cap \left\{y^* \in \R^{q+1} \bigg\vert\; y^*_{q+1} \leq \frac{1}{q-1} \sum_{i=1}^{q-1} y^*_i \right\}.$$
	Since the vertex $t^{q-1}$ of $\bar\D_0$ also belongs to $\bar\D_2$, it is a vertex of $\bar\D_2$ and can be chosen next. Proceeding in this way, we obtain the outer approximation $\bar\D_q$ of $\bar\D$ with corresponding inner approximation
	$$ \bar\P_q = \conv\left\{x^1,\,\ldots,\,x^{q+1}\right\}+\R^{q+1}_+$$
	of $\bar\P$.
	
	Let $\hat t$ be a vertex of $\bar\D_q$ which does not belong to $\bar\D$. An optimal solution of $($P$_1(w(\hat t)))$ is attained in a vertex of $\bar\P$. The points $x^1,\,\ldots,\,x^{q+1}$ cannot be optimal solutions of $($P$_1(w(\hat t)))$ because in this case, $\hat t$ was cut off in an earlier iteration step. Thus $x^0$ must be an optimal solution of $($P$_1(w(\hat t)))$. For $\varepsilon > 0$ being sufficiently small, we have $w(\hat t)> 0$. The only facet of $\bar\P_q$ with (inner) normal vector $w>0$ is $F \colonequals \conv\{x^1,\,\ldots,\,x^{q+1}\}$. Thus, geometric duality implies that $\hat t$ is uniquely defined by $\hat t = \Psi^{-1}(F) = (q+1)^{-1} e$. We have $\varphi(x^0,\hat t)= - \varepsilon$ and thus the algorithm terminates with $\bar\P_{outer} = \bar\P_q$. For sufficiently small $\varepsilon > 0$ we obtain $d_H(\bar\P_{outer},\bar\P)=\varepsilon \|e\|_2= \varepsilon \sqrt{q+1}$. To construct an example of \ref{MOCP}, we define $\P \colonequals \bar\P - \{\frac{1}{q+1} e\}$. From the preceding results, we have $\P_{outer} = \bar\P_{outer}-\{\frac{1}{q+1}e\}$. Then we obtain $d_H(\P_{outer},\P)=\varepsilon \sqrt{q+1}$
\end{example}

We next point out a special property of $\P_{inner}$, when \ref{MOCP} was obtained from \ref{CPP}.

\begin{proposition} \label{prop:dual2}
	Let an arbitrary instance of \ref{CPP} be given. Let $\P_{inner}$ be the result of Algorithm \ref{alg:dual} applied to the associated \ref{MOCP} for some given $\varepsilon > 0$. Then,
	$$ \P_{inner} = (\P_{inner} \cap H) + \R^{q+1}_+.$$
\end{proposition}
\begin{proof}
	From $\P_0 \subseteq \P_{inner} \subseteq \P$, we conclude that $\P_{inner}$ has the recession cone $\R^{q+1}_+$, i.e. $\P_{inner}$ can be expressed by its vertices $\vertex \P_{inner}$ as
	$$ \P_{inner} = \vertex \P_{inner} + \R^{q+1}_+.$$
	 The inclusion $\supseteq$ is now obvious. To prove the converse inclusion we show that each vertex of $\P_{inner}$ belongs to $H$. Let $y$ be a vertex of $\P_{inner}$. Geometric duality (applied to an MOLP with upper image $\P_{inner}$) yields that $F^* \colonequals \Psi^{-1}(\{y\})$ is a $K$-maximal facet of $\D_{outer}$. In Algorithm \ref{alg:dual} we see that such a facet has the form $F^*=\Psi^{-1}(\{\Gamma(x)\})$ for some $x \in X$. Since $\Psi$ is one-to-one, we get $y=\Gamma(x) \in H$.
\end{proof}

Now we prove the error bound for \ref{CPP} solved with Algorithm \ref{alg:dual}. We obtain the same bound as for Algorithm \ref{alg:primal}.

\begin{theorem}\label{Huas:dual}
		Let an arbitrary instance of \ref{CPP} be given. Let $\P_{inner}$ be the result of Algorithm \ref{alg:dual} applied to the associated \ref{MOCP} for some given $\varepsilon > 0$ and let $Y_{inner} \colonequals \pi[\P_{inner} \bigcap H]$. Then $Y_{inner} \subseteq Y$ and
	$$d_H(Y_{inner},Y) \leq \epsilon\sqrt{q^2+q-1}.$$	
\end{theorem}
\begin{proof}
	Let $\bar v$ be an arbitrary point in $Y$. Then $v \colonequals(\bar v^T,-e^T\bar v)^T \in \P \bigcap H$.
	By Proposition \ref{prop:dual}, we have $v + \varepsilon e \in \P_{inner}$. By Proposition \ref{prop:dual2}, there exists $o \in \P_{inner}\cap H$ with $o \leq v+ \epsilon e$. We have $\bar o \colonequals \pi(o) \in Y_{inner}$, $\bar o_i \leq \epsilon + \bar v_i$, $i=1,\ldots,q$ and $o_{q+1}=-e^T \bar o \leq \epsilon + v_{q+1} = \epsilon -e^T\bar v$.
	Thus the point $\bar o - \bar v$ belongs to the polytope $T \colonequals \{t\in\R^{q} \mid -e^T t \leq \epsilon,\; t \leq \epsilon e\}$. Like in the proof of Theorem \ref{prop:primaldis}, we obtain $\|\bar o-\bar v\|_2 \leq \epsilon \sqrt{q^2+q-1}$.
	Since $Y_{inner} \subseteq Y$ and for arbitrary $\bar v \in Y$ there exists $\bar o \in Y_{inner}$ with $\|\bar o-\bar v\|_2 \leq \epsilon \sqrt{q^2+q-1}$, we conclude that $d_H(Y_{inner},Y) \leq \epsilon\sqrt{q^2+q-1}$.
\end{proof}

The bound in Theorem \ref{Huas:dual} is tight for any dimension $q\geq 1$ as shown in the following example, see also Figure \ref{fig_d}.
\begin{figure}[hbt]
\resizebox{0.49\textwidth}{!}{%
\centering
\begin{tikzpicture}[x=45,z=-13,y=35,scale=2.1]
\coordinate (NE) at (1.5,1.3);
\coordinate (SW) at (-.5,-1.5);
\draw[white,fill=white] (NE) circle (0.01);
\draw[white,fill=white] (SW) circle (0.01);
\coordinate (NULL) at (0,0,0);
\coordinate (X) at (1.2,0,0);
\coordinate (Y) at (0,0,1.6);
\coordinate (Z) at (0,1.2,0);
\coordinate (P1) at (0,0,0);
\coordinate (P2) at (1,-1,0);
\coordinate (P3) at (0,-1,1);
\coordinate (B1) at (0,1,0);
\coordinate (B2) at (1.5,-1,0);
\coordinate (B3) at (0,-1,1.5);
\coordinate (R) at (1/3,-2/3,1/3);
\coordinate (RX) at (1/3+1/2,-2/3,1/3);
\coordinate (RY) at (1/3,-2/3+1,1/3);
\coordinate (RZ) at (1/3,-2/3,1/3+1/2);
\draw[thick,->,>=stealth'] (NULL) -- (X) node[anchor=north east]{$y_1$};
\draw[thick,->,>=stealth'] (NULL) -- (Y) node[anchor=north west]{$y_2$};
\draw[thick,->,>=stealth'] (NULL) -- (Z) node[anchor=south]{$y_3$};
\draw[opacity=0, fill opacity=0.5,fill=gray!50!red] (R) -- (RX) -- (RY) -- cycle;
\draw[opacity=0, fill opacity=0.5,fill=gray!50!red] (R) -- (RY) -- (RZ) -- cycle;
\draw[opacity=0, fill opacity=0.5,fill=gray!50!red] (R) -- (RX) -- (RZ) -- cycle;
\draw[red,thick, dashed] (R) -- (RX);
\draw[red,thick, dashed] (R) -- (RY);
\draw[red,thick, dashed] (R) -- (RZ);
\draw[blue,thick,fill opacity=0.5,fill=darkgray!40!cyan] (P1) -- (P2) -- (P3) -- cycle;
\draw[blue,thick,fill opacity=0.5,fill=darkgray!50!cyan] (B1) -- (P1) -- (P2) -- (B2);
\draw[blue,thick,fill opacity=0.5,fill=darkgray!50!cyan] (B2) -- (P2) -- (P3) -- (B3);
\draw[blue,thick,fill opacity=0.5,fill=darkgray!50!cyan] (B3) -- (P3) -- (P1) -- (B1);
\draw[blue, fill=blue] (P1) circle (0.015);
\draw[blue, fill=blue] (P2) circle (0.015);
\draw[blue, fill=blue] (P3) circle (0.015);
\draw[red, fill=red] (R) circle (0.015);
\draw[red] (.5,-.65) node {$\P_{inner}$};
\draw[blue] (.7,-.4) node {$\P$};
\end{tikzpicture}
}
\resizebox{0.49\textwidth}{!}{%
\begin{tikzpicture}[x=45,y=45, scale=2.1]
\coordinate (NE) at (1.5,1.3);
\coordinate (SW) at (-.5,-.5);
\draw[white,fill=white] (NE) circle (0.01);
\draw[white,fill=white] (SW) circle (0.01);
\draw[->, >=stealth'] (-0.02,0) -- (1.1,0) node[right] {$y_1$};
\draw[->, >=stealth'] (0,-0.02) -- (0,1.1) node[above] {$y_2$};
\foreach \x/\xtext in {0,1}
   \draw (\x,0.01) -- (\x,-0.01) node[anchor=north] {$\xtext$};
\foreach \x/\xtext in {0,1}
   \draw (0.01,\x) -- (-0.01,\x) node[anchor=east] {$\xtext$};
\draw[blue,thick,fill opacity=0.5,fill=darkgray!40!cyan] (0,0) -- (1,0) -- (0,1) -- cycle;
\draw[red, fill=red] (1/3,1/3) circle (0.015);
\draw[red] (.4,.43) node {$Y_{inner}$};
\draw[blue] (.6,.15) node {$Y$};
\end{tikzpicture}
}
\caption{Illustration of Example \ref{exam:dual} for $q=2$.}\label{fig_d}
\end{figure}
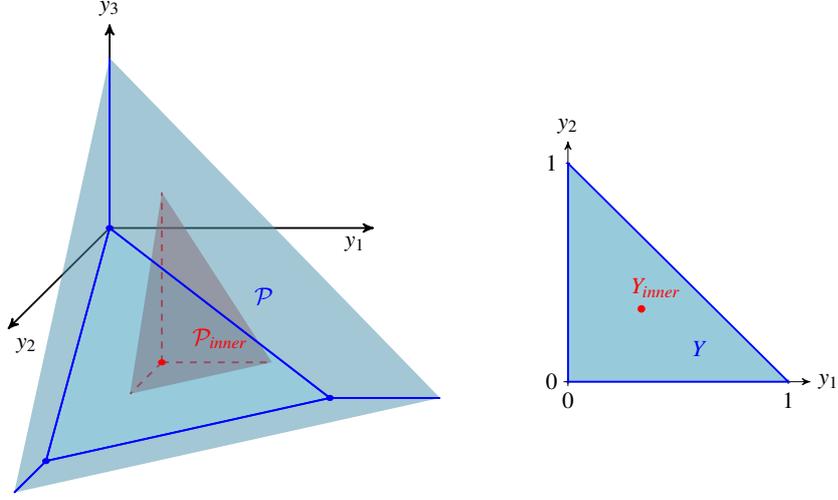

\newpage
\begin{example}\label{exam:dual}
	Let $\varepsilon = \frac{1}{q+1}$ and let
	$$ Y=X=\{ x \in \R^q \mid x \geq 0,\, e^T x \leq 1\}.$$
	The vertices of $Y$ are $0$ and unit vectors $e^i$ for $i\in \{1,\dots, q\}$. A V-representation of $\P$ is given by the extremal directions $e^i$ for $i\in \{1,\dots, q\}$ and the vertices
	 $$ \begin{pmatrix}
	 	\phantom{-}e^1 \\ -1
	 \end{pmatrix} , \ldots,
	  \begin{pmatrix}
	 \phantom{-} e^q \\ -1
	 \end{pmatrix},\;
	  \begin{pmatrix}
	   0 \\ 0
	 \end{pmatrix} $$
	 of $\P$. From geometric duality we obtain an H-representation of $\D$ as
	 $$ \D=\left\{y^* \in \R^{q+1} \bigg\vert\; \forall j \in \{1,\ldots,q\}:\; y^*_j \geq 0,\; \sum_{i=1}^q y^*_i \leq 1,\; y^*_{q+1} \leq y^*_j + \sum_{i=1}^q y^*_i -1,\; y^*_{q+1} \leq 0 \right\}.$$
	 To get the initial outer approximation $\D_0$ of $\D$, we solve $($P$_1(w))$ for $w=\frac{1}{q+1}e$. The set of optimal solutions is $X$. We choose
	 $$x \colonequals \frac{1}{q+1}e \in X,\quad
    \Gamma(x) = \frac{1}{q+1}\begin{pmatrix}
	   \phantom{-}e \\ -q
	 \end{pmatrix} \in \Gamma[X].$$
	  This yieds $\D_0$ with H-representation
	 $$ \D_0=\left\{y^* \in \R^{q+1} \bigg\vert\; \forall j \in \{1,\ldots,q\}:\; y^*_j \geq 0,\; \sum_{i=1}^q y^*_i \leq 1,\; y^*_{q+1} \leq \sum_{i=1}^q y^*_i -\frac{q}{q+1} \right\}$$
	 and vertices
	 $$ \begin{pmatrix}
	 	e^1 \\ \varepsilon
	 \end{pmatrix} , \ldots,
	  \begin{pmatrix}
	 e^q \\ \varepsilon
	 \end{pmatrix},\;
	  \begin{pmatrix}
	   0 \\ \varepsilon-1
	 \end{pmatrix}.$$
	 Since
 	 $$ \begin{pmatrix}
 	 	e^1 \\ 0
 	 \end{pmatrix} , \ldots,
 	  \begin{pmatrix}
 	 e^q \\ 0
 	 \end{pmatrix},\;
 	  \begin{pmatrix}
 	   \phantom{-}0 \\ -1
 	 \end{pmatrix} \in \D
 	 ,$$
	 the algorithm terminates. We obtain
	 $$ \P_{inner} = \{\Gamma(x)\} + \R^{q+1}_+ \quad \text{and} \quad Y_{inner} = \left\{\frac{1}{q+1} e\right\}.$$
	 Thus
	 $$ d_H(Y,Y_{inner}) \geq \left\|e^1 - \frac{1}{q+1} e\right\| = \varepsilon \sqrt{q^2+q-1}.$$
\end{example}

\section{Conclusion} \label{sec:conclusion}
Outer or inner polyhedral approximations $Y_{approx}$ of a convex body $Y$ can be obtained from approximate solutions of a convex projection problem \ref{CPP}. An approximate solution of \ref{CPP} can be obtained by solving an associated multiobjective convex program (MOCP) with the primal or dual Benson type algorithm for a given tolerance $\varepsilon >0$. We have shown that the Hausdorff distance between $Y$ and its polyhedral approximations $Y_{approx}$ obtained by these methods is bounded tightly by $\varepsilon \sqrt{q^2+q-1}$, where $q \geq 2$ is the dimension of $Y$.

\vskip 6mm
\noindent{\bf Acknowledgments}

\noindent   The authors thank both reviewers for their helpful comments. The third author was partially supported by the National Natural Science Foundation of China under Grant No. 12071025.


\end{document}